\documentclass[reqno,11pt]{amsart}
\usepackage{amsmath, amsthm, amscd, amsfonts, amssymb}%
\usepackage{bm,float}
\usepackage{ifpdf}
\ifpdf
\usepackage[hyperindex]{hyperref}
\usepackage{breakurl}
 \usepackage{graphicx}
\usepackage{algorithm}
\usepackage{algorithmic}
\else
\expandafter \ifx \csname dvipdfm \endcsname \relax
\usepackage[hypertex,hyperindex]{hyperref}
\else
\usepackage[dvipdfm,hyperindex]{hyperref}
\fi
\fi
\usepackage{color}

\DeclareMathOperator*{\argmin}{arg\,min}
\DeclareMathOperator*{\diag}{\,diag}

\usepackage{lipsum}
\allowdisplaybreaks

\usepackage[margin=1in]{geometry}
\usepackage{setspace}
\usepackage[foot]{amsaddr}
\begin{document}
\title[PinT methods for American Options]{Parallel-in-Time Iterative Methods for\\Pricing American Options}
\author[X.-M. Gu, et al.]{Xian-Ming Gu, Jun Liu$^{\ast}$, Cornelis W. Oosterlee}  
\address{$^{\ast}$Corresponding author.}
\address{X.-M. Gu \newline
School of Mathematics,
Southwestern University of Finance and Economics,
Chengdu 611130, P.R. China; Bernoulli Institute for Mathematics, Computer Science and Artificial Intelligence, University of Groningen, Nijenborgh 9, P.O. Box 407, 9700 AK Groningen, the Netherlands}
\email{guxianming@live.cn}
\address{J. Liu \newline
Department of Mathematics and Statistics, Southern Illinois University Edwardsville, Edwardsville, IL 62026, USA}
\email{juliu@siue.edu}
\address{C.W. Oosterlee \newline
	Mathematical Institute, Utrecht University, 3584 CS Utrecht, The Netherlands}
\email{c.w.oosterlee@uu.nl}


\thanks{Submitted on \today.}


\begin{abstract}
For pricing American options, 
a sequence of discrete linear complementarity problems (LCPs) or equivalently Hamilton-Jacobi-Bellman (HJB) equations need to be solved in a sequential time-stepping manner. In each time step, the policy iteration or its penalty variant is often applied due to their fast convergence rates. In this paper, we aim to solve for all time steps simultaneously, by applying the policy iteration to an ``all-at-once form" of the HJB equations, where two different parallel-in-time preconditioners are proposed to accelerate the solution of the linear systems within the policy iteration. Our proposed methods are generally applicable for such all-at-once forms of the HJB equation, arising from  option pricing problems with optimal stopping and nontrivial underlying asset models. Numerical examples are presented to show the feasibility and robust convergence behavior of the proposed methodology.
\end{abstract}

\maketitle
\numberwithin{equation}{section}
\newtheorem{theorem}{Theorem}[section]
\newtheorem{lemma}[theorem]{Lemma}
\newtheorem{proposition}[theorem]{Proposition}
\newtheorem{corollary}[theorem]{Corollary}
\newtheorem*{remark}{Remark}

\section{Introduction}
Pricing American-style derivatives has attracted a lot of interest in the academia in the last few decades 
and is of practical importance to the financial industry.
An American option is a financial instrument that gives its buyer the right, but not the obligation, to buy (or sell) an asset at a predetermined price at any time up to a certain terminal time. This  early exercise right, compared with a European option, casts the American option pricing problem into the following nonlinear and comparably intriguing linear complementarity problem (LCP) \cite{Wilmott1993},
\begin{equation}
\begin{cases}
\mathcal{L}V({\bm x},t) \geq 0,\\
V({\bm x},t) \geq V^{*}({\bm x}),\\
\mathcal{L}V({\bm x},t)\cdot\left[V({\bm x},t) - V^{*}({\bm x})\right] = 0,
\end{cases}
\label{eq1.1}
\end{equation}
for $({\bm x},t)\in\mathbb{R}^{d}_{+}\times[0,T)$ $(d = 1,2, \ldots)$; with the terminal condition $V({\bm x},T) = V^{*}({\bm x})$ denoting the payoff function of the option (e.g.,  $\max\{K-{\bm x},0\}$ for a put option). ${\bm x}$ and $\mathcal{L}$ are, respectively,  the independent variables in space and the linear differential operator originating from the model assumptions for the underlying price, e.g. the  Black-Scholes (B-S) operator \cite{Black73,Merton73}, two-asset B-S operator \cite{heidarpour2018}, or the Heston stochastic volatility model \cite{Haentjens15,Le2012}, in this study. The proposed methods can be applied to other option pricing models after appropriate modification.

In practical computations, we need to truncate the computational domain $\Omega\subset\mathbb{R}^d$ and add suitable initial-boundary conditions. Three, out of many, methodologies are prominent for pricing American options: regression Monte Carlo simulation, see~\cite{broadie1996american,longstaff2001valuing}, numerical integration and transform methods, and numerical solution of partial differential equations (PDEs). Each method has distinct advantages and limitations. Within this landscape, the finite difference method and Monte Carlo simulation emerge as key contenders. The former is intuitive and offers a systematic approach to  numerically solving the option pricing equations, while the latter employs randomized sampling for scenario generation, backward regression techniques under optimal stopping scenarios and can easily be generalized to high-dimensional settings. 

In the PDE context, options with early-exercise features give rise to free boundary PDE problems. Traditional methods such as lattice-based binomial tree models remain reliable for Bermudan and American options~\cite{borovkova2012american}, grounded in fundamental financial principles but their generalisation regarding asset dynamics is nontrivial, see \cite{vellekoop2009tree}. Numerical PDE solutions for American options~\cite{forsyth2007numerical,ikonen2008efficient,hout2016application} have been developed for different formulations of the free boundary problem, like for LCPs~\cite{oosterlee2003multigrid}, parabolic variational inequalities~\cite{jaillet1990variational}, or penalty methods~\cite{zvan1998penalty,forsyth2002quadratic} and policy iteration methods~\cite{Reisinger12,arregui2017pde}. Recent advancements in transform methods for early-exercise options include Fourier-based techniques, for example, the COS method~\cite{fang2009pricing}, leveraging Fourier-cosine series expansions for efficient computations and wavelet expansion methods, like the Shannon Wavelet Inverse Fourier Technique (SWIFT)~\cite{maree2017pricing}, complementing traditional approaches. In this paper, we will develop fast iterative solvers for the well-known PDE models.

For efficient numerical treatment, the above LCP (\ref{eq1.1}) is often interpreted as the following Hamilton-Jacobi-Bellman (HJB) equation (also called the obstable problem in literature)
\begin{equation}
	\begin{cases}
		\min\left\{\partial_t u - \mathcal{\widehat{L}} u, u - \phi(s)\right\} = 0, & t\in (0,T],\\
		u(s,0) = \phi(s),~s\in\Omega~{\rm(1D)~or}~u(s,v,0)=\phi(s),~(s,v)\in\Omega~({\rm 2D}),\\
		{\rm boundary~conditions~(BCs)},
	\end{cases}
	\label{eq1.4}
\end{equation}
where $u(s, t)$ or $u(s,v,t)$ is the value of an American put option with striking price $K$, i.e., the holder can receive a given pay-off function
$\phi(s)$ at the expiry date $T$. Moreover, a time-reverse transformation $u(s,t) = V(s,T - t)$ or $u(s,v,t) = V(s,v,T - t)$ is also utilized in the derivation of the above HJB equation, where the operator $\mathcal{\widehat{L}}$ corresponding to the operator $\mathcal{L}$ in Eq. (\ref{eq1.1}).

In the following, put options are considered. Their price at the expiry is given by the pay-off function $\phi(s) = \max\{K - s, 0\}$. As the equations are solved backward in time, this leads to the initial condition
\begin{equation}
u(s,0) = \phi(s)~~{\rm and}~~u(s,v,0) = \phi(s)
\end{equation}
for 1D and 2D models, respectively. In order to compute an approximate solution, the infinite domain is truncated at $s = s_{max}$ and $v = v_{max}$, where $s_{max}$ and are sufficiently large so that the error due to the truncation is negligible. For the detail information of BCs, we will describe it in the specific examples involving the above PDE models (cf. Section \ref{sec4}). 

With  the  advent  of  massively  parallel  processors, parallelizable numerical methods for solving the HJB equations (\ref{eq1.4}) become 
increasingly meaningful, but there are still few studies  on this topic \cite{Clevenhaus22}. More precisely, instead of solving such LCPs via time-stepping in a sequential manner, we propose a novel numerical treatment named the parallel-in-time (PinT\footnote{We refer to  the  website  \url{http://parallel-in-time.org} for  various  PinT  algorithms  and applications.}) method that solves for all unknowns simultaneously by constructing a large LCP which is composed of smaller LCPs at each time level. More specifically, we will exploit the policy iteration to solve this large all-at-once LCP and each iteration requires to solve a linear system with a much higher computational cost. In order to alleviate this high workload, we design a kind of parallel preconditioning techniques for accelerating a Krylov subspace solver to the above linear system. By making  use  of  the  obtained (stable) diagonalization and the fast Fourier transform (FFT) for factorizing the involved $\alpha$-circulant matrices, for any input  vector ${\bm r}$,  we  can  compute  the  preconditioning  step $\left[\mathcal{P}_\alpha^{(k)}\right]^{-1}{\bm r}$ very efficiently via the well-studied three-step diagonalization technique \cite{lin2021all,gander21paradiag}, which yields an effective parallel implementation across all time grid points. Interestingly, we numerically observe that the number of all-at-once policy iteration is independent of the number of coupled time steps. We were not able to fully justify this observation yet, but it seems to be related with a fascinating conclusion \cite[Corollary 2.6]{Reisinger12} that the overall number of policy iterations in a typical time-stepping scheme is bound by $(N_t+N_s)$ when the solution from the
previous time step is used as the initial value for current time step. Our all-at-once policy iteration essentially marches $N_t$ different time steps forward simultaneously (see Figure \ref{FigEx1_IndexPattern}) and hence its convergence rate only depends on $N_s$.

The reminder of this paper is organized as follows. In Section \ref{sec2}, we present the semi-discrezation of HJB equations and construct the time-stepping method using the policy iteration for discretized HJBs. Section \ref{sec3} introduce the all-at-once form of HJBs and then solve it via combining the policy iteration and the PinT preconditioned Krylov subspace method. The construction of the PinT preconditioning has been described in details. In Section \ref{sec4}, numerical experiments involving three American option pricing models are reported to show the effectiveness of our proposed method. Finally, some conclusions are drawn in the last Section \ref{sec5}.
\section{Semi-discretized HJB equations}
\label{sec2}
In this section, we consider to compute the price (or value) of an American put option governed via the HJB equation \ref{eq1.4} (see \cite{Pham1997} for the existence and uniqueness results in the viocosity solution framework).
At this stage, we define $t_n = n\tau$ with $\tau = T/N_t$, where $N_t$ is a positive integer. Let $L_h\in \mathbb{R}^{N_s\times N_s}$ be the spatial discretization matrix of $\mathcal{\widehat{L}}$, via a standard finite difference method (or finite element method).
Then we obtain the semi-discretized system of the following form
\begin{equation}
\begin{cases}
\min\left(\frac{d{\bm u}(t)}{dt} - L_h{\bm u}(t) - {\bm g}, {\bm u}(t) - {\bm \phi}\right) = 0,\\
{\bm u}(0) =  {\bm \phi},
\end{cases}
\end{equation}
where ${\bm u}(t)\in\mathbb{R}^{N_s}$ is a lexicographically ordered vector collecting the approximate solution of $u(\cdot, t)$ over all the spatial grid points. the vector ${\bm g}$ is obtained by discretizing the suitable boundary conditions of the model \eqref{eq1.4}. Moreover, for ${\bm x}, {\bm y}\in\mathbb{R}^{N_s}$ the notation $\min\{{\bm x},{\bm y}\}$ denotes the vector of components $\min\{x_i,y_i\}$. Let ${\bm u}^n={\bm u}(t_n)$, we use the first-order accurate implicit Euler scheme to discretize the above problem as follows (with ${\bm u}^0={\bm \phi}$ from the initial condition):
\begin{equation}
\min\left(\frac{{\bm u}^n - {\bm u}^{n-1}}{\tau} - L_h{\bm u}^n - {\bm g}, {\bm u}^{n} - {\bm \phi}\right) = 0,\quad n = 1,2,\cdots,N_t,
\label{eq2.3}
\end{equation}
which is usually solved by the so-called policy iteration (or Howard's algorithm) in each time level $n = 1,2,\cdots,N_t$:
\begin{equation}
\min\left\{\left(\frac{1}{\tau}I_s - L_h\right){\bm u}^n - \left(\frac{1}{\tau}{\bm u}^{n-1} + {\bm g}\right), {\bm u}^{n} - {\bm \phi}\right\} = {\bm 0},
\label{eq2.4}
\end{equation}
where $I_s$ is an identity matrix of order $N_s$. Clearly, this time stepping algorithm is sequential.

For convenience, we now briefly recall the standard policy iteration for solving the following discrete HJB equation or obstacle problem
\begin{equation}
\min\{A{\bm x} - {\bm b}, {\bm x} - {\bm c}\} = {\bm 0},
\end{equation}
which has a unique solution if and only if $A$ is a P-matrix (i.e., every principal minor is positive). 
In $k$-th policy iteration with a given policy ${\bm x}^{k-1}$, one needs to solve
for the updated policy ${\bm x}^{k}$ from the following sequence of linear system
\begin{equation} \label{policyAk}
	A^{(k)} {\bm x}^{k}={\bm b}^{(k)},
\end{equation}
where
\begin{equation}
	A^{(k)}:=I+\Phi^{(k)}(A-I),\qquad {\bm b}^{(k)}:= {\bm c}+\Phi^{(k)}({\bm b}-{\bm c}),
\end{equation}
with the 0--1 diagonal matrix $\Phi^{(k)}$ given by
\begin{equation}
	\Phi^{(k)}_{i,i}:=
	\begin{cases}
		1,  & \text{if}\ {(A{\bm x}^{k-1}-{\bm b})_i\le ({\bm x}^{k-1}-{\bm c})_i},\\
		0, &  \text{otherwise.}
	\end{cases}
\label{eq2.8}
\end{equation}
The detail convergence analysis of such a policy iteration for the above sequence of HJB (or obstacle) problems has been discussed in \cite{Bokanowski09,Reisinger12}, with the assumption of $A$ being a nonsingular M-matrix (i.e. $A^{-1}\ge 0$).
In the literature, the sequence of linear systems (\ref{policyAk}) is typically solved by a sparse direct solver (such as backslash solver in MATLAB), which can become prohibitively expensive when the system size gets large in multidimensional cases. 
In this paper, we will propose efficient PinT preconditioners for solving (\ref{policyAk})  by Krylov subspace methods, such as preconditioned GMRES.
We point out that the discretized matrix $A$ is sparse and highly structured, 
but its structure is heavily distorted in $A^{(k)}$, which leads to some essential difficulty in developement of fast iterative solvers.
%
\section{Parallel-in-time policy iteration}
\label{sec3}
\subsection{All-at-once form of LCPs}
In this section, we explore a possibility for PinT computation of (\ref{eq2.3}) by using a modified diagonalization technique, which is originally proposed by Maday and R{\o}nquist \cite{Maday2008}. The idea can be briefly introduced as follows. For (\ref{eq2.4}), we always solve the following sequence of obstacle problems:
\begin{equation}
\begin{cases}
\min\left\{\left(\frac{1}{\tau}I_s - L_h\right){\bm u}^1 - \left(\frac{1}{\tau}{\bm u}^{0} + {\bm g}\right), {\bm u}^{1} - {\bm \phi}\right\} = {\bm 0},\\
\min\left\{\left(\frac{1}{\tau}I_s - L_h\right){\bm u}^2 - \left(\frac{1}{\tau}{\bm u}^{1} + {\bm g}\right), {\bm u}^{2} - {\bm \phi}\right\} = {\bm 0},\\
\qquad\qquad\qquad\qquad\qquad \vdots\\
\min\left\{\left(\frac{1}{\tau}I_s - L_h\right){\bm u}^{N_t} - \left(\frac{1}{\tau}{\bm u}^{N_t-1} + {\bm g}\right), {\bm u}^{N_t} - {\bm \phi}\right\} = {\bm 0},
\end{cases}
\label{eq3.1x}
\end{equation}
in the one-by-one mode. However, instead of computing the solutions of Eq. (\ref{eq3.1x}) step-by-step, we try to get the solutions simultaneously by solving the following large-scale obstacle problem:
\begin{equation}\small
\min\left\{
\underbrace{
\begin{bmatrix}
\frac{1}{\tau}I_s - L_h \\
-\frac{1}{\tau}I_s & \frac{1}{\tau}I_s - L_h\\
& \ddots & \ddots\\
&& -\frac{1}{\tau}I_s & \frac{1}{\tau}I_s - L_h
\end{bmatrix}
}_{\mathcal{M}}
\underbrace{
\begin{bmatrix}
{\bm u}^1\\
{\bm u}^2\\
\vdots\\
{\bm u}^{N_t}
\end{bmatrix}
}_{\bm v}
-
\underbrace{
\begin{bmatrix}
\frac{1}{\tau}{\bm u}^0 + {\bm g}\\
	{\bm g}\\
	\vdots\\
	{\bm g}
\end{bmatrix}
}_{\bm f},
\underbrace{
\begin{bmatrix}
	{\bm u}^1\\
	{\bm u}^2\\
	\vdots\\
	{\bm u}^{N_t}
\end{bmatrix}
}_{\bm v}-
\underbrace{
\begin{bmatrix}
	{\bm \phi}\\
	{\bm \phi}\\
	\vdots\\
	{\bm \phi}
\end{bmatrix}
}_{\bm \Phi}
\right\} =
\underbrace{
\begin{bmatrix}
	{\bm 0}\\
	{\bm 0}\\
	\vdots\\
	{\bm 0}
\end{bmatrix}
}_{\bm 0}.
\label{eq3.1}
\end{equation}
For clarity, we have introduced the following notations
\begin{equation}
\mathcal{M}=\begin{bmatrix}
\frac{1}{\tau}I_s - L_h \\
-\frac{1}{\tau}I_s & \frac{1}{\tau}I_s - L_h\\
& \ddots & \ddots\\
&& -\frac{1}{\tau}I_s & \frac{1}{\tau}I_s - L_h
\end{bmatrix},\quad
{\bm v} =
\begin{bmatrix}
	{\bm u}^1\\
{\bm u}^2\\
\vdots\\
{\bm u}^{N_t}
\end{bmatrix},~~
{\bm f} = \begin{bmatrix}
\frac{1}{\tau}{\bm u}^0 + {\bm g}\\
{\bm g}\\
\vdots\\
{\bm g}
\end{bmatrix},~~
{\bm \Phi} =
\begin{bmatrix}
	{\bm \phi}\\
{\bm \phi}\\
\vdots\\
{\bm \phi}
\end{bmatrix},
\label{eq3.2}
\end{equation}
then we reformulate the above problem (\ref{eq3.1}) in all-at-once form:
\begin{equation}
	\min\left(\mathcal{M}{\bm v} - {\bm f}, {\bm v} - {\bm \Phi}\right) = {\bm 0},~~\mathcal{M} = B\otimes I_s -  I_t\otimes L_h,
	\label{eq3.3}
\end{equation}
where the matrix $B = \frac{1}{\tau}tridiag(-1,1,0)\in\mathbb{R}^{N_t\times N_t}$
and $I_t$ is an identity matrix of order $N_t$.

Before we solve the resulting all-at-once problem (\ref{eq3.1}), the following equivalence property between the solutions of (\ref{eq3.1}) and \eqref{eq2.3} can be shown, which lays the foundation of our approach.
\begin{proposition}
The vectors ${\bm u}^1, {\bm u}^2, \cdots, {\bm u}^{N_t}$ are the solutions of obstacle problems (\ref{eq2.4}) if and only if the vector ${\bm v}$ in  \eqref{eq3.2} is the solution of the all-at-once problem \eqref{eq3.1}.
\end{proposition}
\begin{proof} First of all, when the vectors ${\bm u}^1, {\bm u}^2, \cdots, {\bm u}^{N_t}$ are the solutions of obstacle problems (\ref{eq2.4}), it means that \eqref{eq2.4} holds for $n = 1,2,\cdots,N_t$. According to the definition of $\min\{{\bm x}, {\bm y}\}$ involving two vectors ${\bm x}$ and ${\bm y}$, we collect them as an equivalent form \eqref{eq3.1}. On the other hand, if the vector ${\bm U}$ in \eqref{eq3.2} is the solution of the all-at-once problem \eqref{eq3.1}, we can equivalently decouple it as a sequence of obstacle problems (\ref{eq2.3}) by using the component-wise definition of ``$\min$".
\end{proof}
Since the convergence of our used policy iteration depends on the coefficient matrix being an $M$-matrix, the following property shows the all-at-once large coefficient matrix is still an $M$-matrix.
\begin{proposition}
If $-L_h$ is an $M$-matrix, then the matrix $\mathcal{M}$ is also a nonsingular $M$-matrix.
\end{proposition}
\begin{proof}
Since $-L_h$ is an $M$-matrix, so does the matrix $\frac{1}{\tau}I - L_h$ and then $\mathcal{M}$ is a $Z$-matrix \cite{Berman94}. On the other hand, all the eigenvalues of $\mathcal{M}$ are equal to those of the matrices $\frac{1}{\tau}I - L_h$ on the diagonal. Meanwhile, since every eigenvalue of the matrix $\frac{1}{\tau}I - L_h$ has positive real part, which implies that all the eigenvalues of the matrix $\mathcal{M}$ have  positive real part. With all that said, we can conclude that $\mathcal{M}$ is a nonsingular $M$-matrix.
\end{proof}
\begin{remark}
There are many discretizations for the spatial operator $\mathcal{L}$ in Eq. (\ref{eq1.4}) and the corresponding discrete matrix $-L_h$ is indeed an $M$-matrix, refer e.g., to \cite{Cen2010,Shen2023,Le2012,Hackbusch17,itkin2017} for details.
\end{remark}

To apply the policy iteration for solving (\ref{eq3.2}), we need to solve the following linear system
\begin{equation}
\mathcal{M}^{(k)}{\bm v}^{(k)} = {\bm f}^{(k)},
\label{eq3.4}
\end{equation}
where
\begin{equation}
\mathcal{M}^{(k)} = \mathcal{I} + \Theta^{(k)}(\mathcal{M} - \mathcal{I}),~~
\mathcal{I} = I_t\otimes I_s,~~{\bm f}^{(k)} = {\bm \Phi} + \Theta^{(k)}({\bm f} - {\bm \Phi})
\end{equation}
and $\Theta^{(k)}$ is a (block) diagonal matrix whose diagonal elements are 1 or 0, cf. the definition in Eq. \eqref{eq2.8}. Moreover, it noted that the coefficient matrix $\mathcal{M}^{(k)}$ is a block lower bidiagonal matrix, then if a Gaussian elimination based block forward substitution method \cite{golub2013matrix} is used to solve Eq. \eqref{eq3.4}, the computational complexity is of order $\mathcal{O}(N_tN^{3}_s + N_t N^{2}_s)$. This is still expensive, especially for high-dimensional models with a fine mesh. Considering the special sparse structure of $\mathcal{M}^{(k)}$, some iterative solvers (e.g., Krylov subspace methods) which only depend on the coefficient matrix-vector products at each iteration are preferable for Eq. \eqref{eq3.4}, see e.g. \cite{Mrkaic02,saad2003iterative}. However, without the use of preconditioning the Krylov subspace solver for Eq. \eqref{eq3.4} may not converge, or may take many iterations to return an acceptably accurate solution. In the best case, preconditioned Krylov iterative methods can substantially reduce the storage and operation costs as opposed to direct methods, while achieving greatly improved convergence rates that may often be proved a priori.
%
%
%
\subsection{PinT solvers}
In this subsection, we will design the PinT preconditioning techniques for accelerating the policy iteration to all-at-once problem (\ref{eq3.3}). In fact, the main difficulty of designing such preconditioners is how to well approximate a sequence of block diagonal matrices $\Theta^{(k)}$, which are comprised of both the time and space discretizations (especially the latter), it forbids us to use a unified framework to realize this target but we still can describe the following specific ideas for 1D and 2D models, respectively.
\subsubsection{1D model: nearest Kronecker product approximation.}
At this stage, we would like to design an effective preconditioning technique for accelerating Krylov subspace solvers applied to Eq. \eqref{eq3.4}. We first write $\Theta^{(k)}$ into a block diagonal form of $N_t$ blocks
\begin{equation}
	\Theta^{(k)}=\mathrm{blkdiag}\left(\Theta_1^{(k)},\Theta_2^{(k)},\cdots,\Theta_{N_t}^{(k)}\right).
\end{equation}

Based on the need of PinT solver, we desire to have a Kronecker product approximation
\begin{equation} \label{KronTheta}
	\Theta^{(k)}\approx I_t\otimes \Psi^{(k)},
\end{equation}
where the choice of identity matrix $I_t$ is due to $\Theta^{(k)}$ being a 0--1 diagonal matrix and $\Psi^{(k)}$ will be determined as below. To motivate the above Kronecker product approximation (\ref{KronTheta}), in Figure \ref{FigEx1_IndexPattern} we plot the nonzero pattern of the following temporally reshaped matrix: 
\begin{equation} \label{HatTheta}
\texttt{Tmat}(\Theta^{(k)}):=\left[\diag(\Theta^{(k)}_1),\diag(\Theta^{(k)}_2),\cdots,\diag(\Theta^{(k)}_{N_t})\right],
\end{equation}
where $\diag(\Theta^{(k)}_j)$ denotes a column vector of the main diagonal elements of $\Theta^{(k)}_j$. 
\begin{figure}[t]
	\begin{center} 
		\includegraphics[width=1\textwidth]{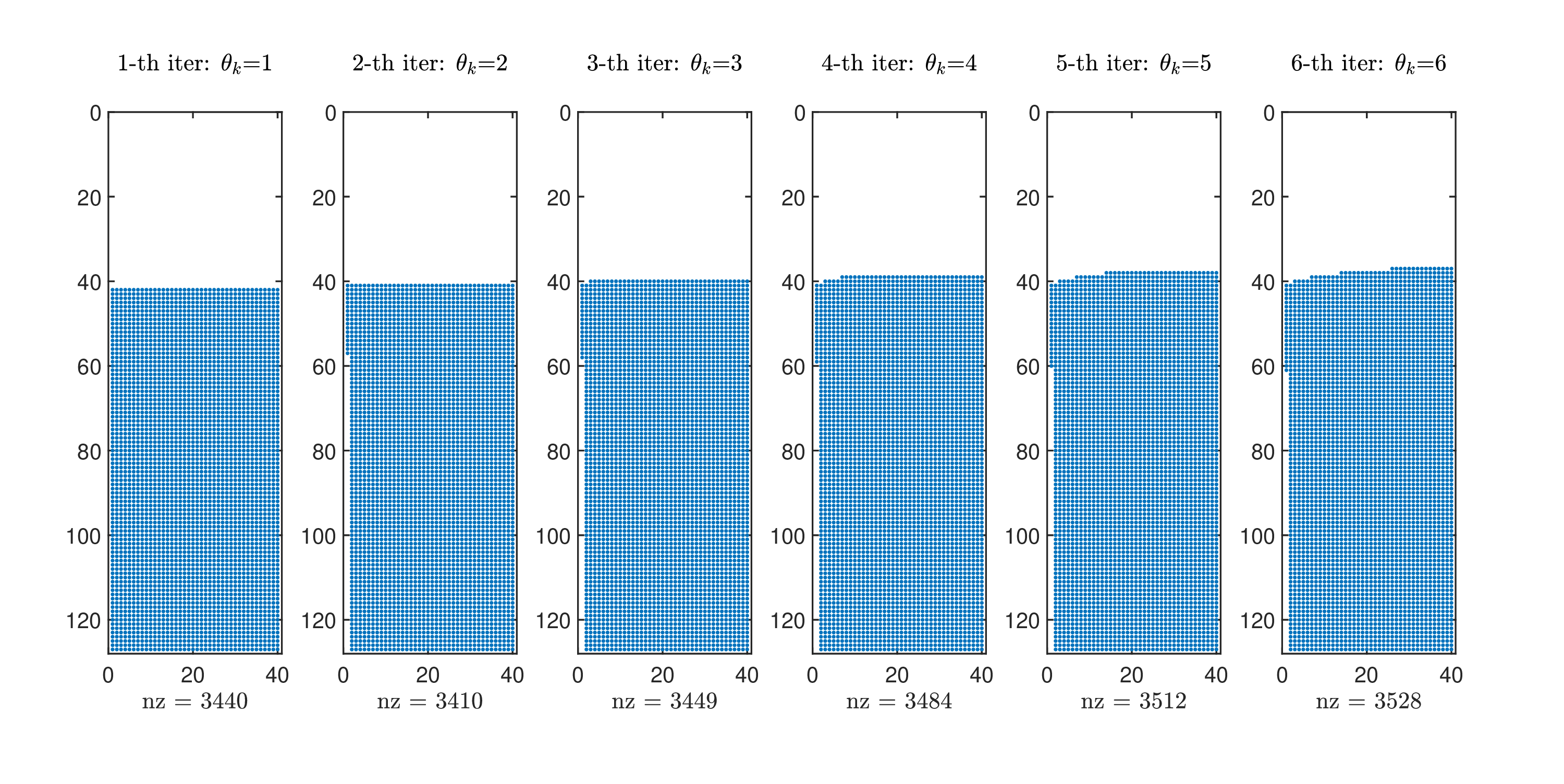}
	\end{center}
	\caption{the  nonzero pattern (blue dot=1, white blank=0) of the  reshaped matrix $\texttt{Tmat}(\Theta^{(k)})$ at each PinT-Policy iteration (Example 1 with $N_s=128$, $N_t=40$ and $\alpha=10^{-8}$). }	
	\label{FigEx1_IndexPattern}
\end{figure}
Clearly, Figure \ref{FigEx1_IndexPattern} shows that almost all columns of $\widehat\Theta^{(k)}$ can be approximated by the same column vector, which is exactly what the Kronecker product approximation (\ref{KronTheta}) is trying to achieve. In other words, all diagonal blocks $\Theta^{(k)}_j$ can be approximated by the same diagonal matrix $\Psi^{(k)}$. More specifically, we are essentially trying to obtain rank-one approximation in the following form
\begin{equation} \label{HatPsi}
	\texttt{Tmat}(\Theta^{(k)})\approx \texttt{Tmat}(I_t\otimes \Psi^{(k)}) =\left[\diag(\Psi^{(k)}),\diag(\Psi^{(k)}),\cdots,\diag(\Psi^{(k)})\right].
\end{equation}
To demonstrate the low-rank structure of $\texttt{Tmat}(\Theta^{(k)})$, we also report in Figure \ref{FigEx1_IndexPattern} its small rank   
\begin{equation}
	\theta_k:=\mathrm{rank}\left(\texttt{Tmat}(\Theta^{(k)})\right)\ll \min(N_t,N_s),
\end{equation}
which explains why the rank-one matrix $\texttt{Tmat}(I_t\otimes \Psi^{(k)})$ can possible provide a good approximation. How fast the 0-1 interface in the nonzero pattern of $\texttt{Tmat}(\Theta^{(k)})$ changes across the policy iterations is determined by the free optimal exercise boundary that depends on the problem parameters.

The diagonal matrix $\Psi^{(k)}$ can be determined from the nearest Kronecker product approximation (NKPA) formulation \cite{VanLoan1993,Liu2022}
\begin{equation}
	\Psi^{(k)}=\argmin_{D\ \text{is diagonal}}\left\|\Theta^{(k)}- I_t\otimes D\right\|_F,
\end{equation}
where $\Psi^{(k)}$ has the following explicit formula (i.e. the average of all diagonal blocks)
\begin{equation}
	\Psi^{(k)}= \frac{1}{N_t}\sum_{j=1}^{N_t}{\Theta_j^{(k)}}.
\end{equation} 
It is interesting to point out the following connection
\begin{equation}
	 \left\|\Theta^{(k)}- I_t\otimes D\right\|_F
 = \|\texttt{Tmat}(\Theta^{(k)})-\texttt{Tmat}(I_t\otimes D)\|_F,
\end{equation}
which implies the above NKPA is essentially a special structured rank-one approximation of $\Theta^{(k)}$.
Another meaningful choice is to take  the most frequently occurring value (0 or 1), that is
\begin{equation}
	\left(\Psi^{(k)}\right)_{i,i}= \mathrm{mode} \left\{(\Theta_j^{(k)})_{i,i}\right\}_{j=1}^{N_t},
\end{equation}
which avoids averaging. Numerically, we found this also works well but slightly less robust. 

In order to utilize the block $\alpha$-circulant type matrix, we establish the preconditioner as follow:
\begin{equation}
\mathcal{P}_\alpha^{(k)} = \mathcal{I} + (I_t\otimes \Psi^{(k)})(\mathcal{M}_{\alpha} - \mathcal{I}),\quad
\mathcal{M}_{\alpha} = B^{(\alpha)}\otimes I_s - I_t\otimes L_h,\quad \alpha\in(0,1),
\end{equation}
where $B^{(\alpha)}$ is an $\alpha$-circulant matrix approximate the matrix $B$ as follow:
\begin{equation*}
B^{(\alpha)} = \frac{1}{\tau}
\begin{bmatrix}
	1 &  & & -\alpha \\
	-1 & 1 &  & \\
	& \ddots & \ddots & \\
	&  & -1 & 1
\end{bmatrix} \in \mathbb{R}^{N_t \times N_t}.
\end{equation*}

In the following, we outline the following implementation of our used
policy iteration with PinT preconditioning for all-at-once LCP (\ref{eq3.3}), where the preconditioner $\mathcal{P}_\alpha^{(k)}$ is our main focus.
\begin{algorithm}
	\caption{Policy iteration with PinT preconditioning for \eqref{eq3.3}}
	\begin{algorithmic}
		\REQUIRE the matrices $B$, $L_h$ and the vectors ${\bm f}, {\bm  \Phi}$
		\ENSURE the approximate solution ${\bm v}$
		\STATE Choose the initial guess ${\bm v}^{0}$
		\FOR{$k = 0,1,\cdots$,}
		\STATE Solve the inner system \eqref{eq3.4} via a Krylov subpace solver with the preconditioner $
		\mathcal{P}_\alpha^{(k)}$,
		\IF{$k\geq 1$}
		\STATE ${\bm v}^{k} = {\bm v}^{k-1}$, Stop
		\ELSE
		\STATE For $i = 1,2,\cdots,N_sN_t$, we take
		\[\Theta^{(k+1)}_{i,i}=
		\begin{cases}
			1, & {\rm if}~ \left(\mathcal{M}{\bm v}^{k} - {\bm f}\right)_i\leq
			\left({\bm v}^{(k)} - {\bm \Phi}\right)_i\\
			%
			%
			%
			0, & {\rm otherwise}.
		\end{cases}
		\]
		\ENDIF
		\ENDFOR
	\end{algorithmic}
\end{algorithm}

Let $\mathbb{F} = \left[ \frac{\omega^{j k}}{\sqrt{N_t}}\right]_{j,k = 0}^{N_t - 1}$ be the discrete
Fourier matrix, where $\omega = e^{-2 \pi \mathrm{\mathbf{i}}/N_t} (\mathrm{\mathbf{i}} = \sqrt{-1})$.
We know that $\alpha$-circulant matrix $B^{(\alpha)}$ can be diagonalized as
$B^{(\alpha)} = V_\alpha D^{(\alpha)} V_\alpha^{-1}$ \cite{bini2005}, with
\begin{equation*}
	V_\alpha = \Lambda_\alpha^{-1} \mathbb{F}^* \quad \mathrm{and} \quad
	D^{(\alpha)} = \mathrm{diag}\left(\sqrt{N_t}\, \mathbb{F}\, \Lambda_\alpha\, B^{(\alpha)}(:,1) \right)
	= \mathrm{diag}\left( \lambda_1,\ldots,\lambda_{N_t} \right),
\end{equation*}
where $\Lambda_\alpha = \mathrm{diag}\left(1, \alpha^{\frac{1}{N_t}}, \ldots, \alpha^{\frac{N_t - 1}{N_t}} \right)$, `*' denotes the conjugate transpose of a matrix,
$B^{(\alpha)}(:,1)$ is the first column of $B^{(\alpha)}$
and
\begin{equation}
\lambda_k = 1 - \alpha^{\frac{1}{N_t}} e^{\frac{2 (k - 1) \pi \mathrm{\mathbf{i}}}{N_t}}~(k = 1, \ldots, N_t).
\label{eq3.13}
\end{equation}
The following property of the eigenvalues $\lambda_k$ in Eq. (\ref{eq3.13}) of the matrix $B^{(\alpha)}$ will be used later.
\begin{proposition}
For $\alpha \in (0,1)$, the real part of $\lambda_{k}$ is positive, i.e.,
$\Re{\rm e}(\lambda_{k}) > 0$.
\label{pro3.3}
\end{proposition}
\begin{proof}
The assertion can be verified directly by noting the Euler formula $e^{\mathbf{i}\theta}=\cos(\theta)+\mathbf{i}\sin(\theta)$.
\end{proof}

According to the property of Kronecker product, we can easily factorize
$\mathcal{P}_\alpha^{(k)}$ as
\begin{equation}
\begin{split}
\mathcal{P}_\alpha^{(k)} & = B^{(\alpha)}\otimes \Psi^{(k)} - I_s\otimes\left[\Psi^{(k)}
L_h + (\Psi^{(k)} - I_s)\right]\\
& = (V_{\alpha}\otimes I_s)\left[D^{(\alpha)}\otimes \Psi^{(k)} -
I_t\otimes \left[\Psi^{(k)}
L_h + (\Psi^{(k)} - I_s)\right]
\right]\left(V^{-1}_{\alpha}\otimes I_x\right),
\end{split}
\end{equation}
and this implies that when the preconditioned Krylov subspace solver is applied to Eq. (\ref{eq3.4}), we can solve the preconditioned sub-system
${\bm z} = \left[\mathcal{P}_\alpha^{(k)}\right]^{-1}{\bm r}$ via three steps:
\begin{equation}
\begin{cases}
{\bm z}_1 = \left(V^{-1}_{\alpha}\otimes I_x\right){\bm r} = (\mathbb{F}\otimes I_s)\left[(\Lambda^{-1}_{\alpha}\otimes I_s){\bm r}\right], &{\rm \text{Step-(a)}},\\
\left[\lambda_n\Psi^{(k)} - \left(\Psi^{(k)}L_h + (\Psi^{(k)} - I_s)\right)\right] {\bm z}_{2,n} = {\bm z}_{1,n},\quad n = 1,2,\cdots,N_t,&{\rm \text{Step-(b)}},\\
{\bm z} = (V_{\alpha}\otimes I_s){\bm z}_2
= (\Lambda_{\alpha}\otimes I_s)\left[(\mathbb{F}^{*}\otimes
I_x){\bm z}_2\right],&{\rm \text{Step-(c)}},
\end{cases}
\label{eq3.15}
\end{equation}
where ${\bm z}_j = \left[{\bm z}^{\top}_{j,1},{\bm z}^{\top}_{j,2},\cdots,{\bm z}^{\top}_{j,N_t}
\right]^{\top}$ (with $j = 1,2$) and $\lambda_n$ is the $n$-th eigenvalue of $B^{(\alpha)}$. The first and third steps only concern matrix-vector multiplications and can be implemented efficiently by using the fast Fourier transforms (FFTs). The major computation cost lies in the second Step-(b), but fortunately this step is fully parallel for the $N_t$ time steps.
For 1D problems, the linear systems in  Step-(b) are tridiagonal and hence they can be solved by fast Thomas solver with $\mathcal{O}(N_s)$ complexity.
However, for 2D/3D problems, the computational cost of solving sparse linear systems in Step-(b)  can be still high, unless more efficient direct or iterative solvers can be designed.
\begin{proposition}
If the spatial matrix $L_h$ is negative semi-definite, then the preconditioner $\mathcal{P}_\alpha^{(k)}$ is well-defined, i.e., it is invertible.
\label{pro3.4}
\end{proposition}
\begin{proof}
According to the definition of the matrix $\mathcal{P}_\alpha^{(k)}$, we need to show all the coefficient matrices of Eq. (\ref{eq3.15}) are nonsingular. For the 0-1 matrix $\Psi^{(k)}$, there is a permutation matrix $P$ such that 
\begin{equation}
P\Psi^{(k)}P^{\top} = 
\begin{bmatrix}
I_1 \\
 & O
\end{bmatrix},\quad PP^\top = I_s: = \begin{bmatrix}
I_1 \\
& I_2
\end{bmatrix},
\end{equation}
where $I_1$ is an identity matrix of order $n_1$ (equal to the number of 1's in the matrix $\Psi^{(k)}$). Hence
\begin{equation}
P(\lambda_nI_s-L_h)P^{\top} = 
\begin{bmatrix}
L_{11} & L_{12}\\
L_{21} & L_{22}
\end{bmatrix},
\end{equation}
which is positive definite according to Proposition \ref{pro3.3}. At this stage, we have
\begin{equation}
\begin{split}
P\left[I_s + \Psi^{(k)}((\lambda_n I_s - L_h) - I_s)\right]P^{\top} & =  
I_s + P\Psi^{(k)}P^{\top} P\left[(\lambda_nI_{s}-L_h) - I_s\right]P^{\top}\\
& = I_s + \begin{bmatrix}
	I_1 \\
	& O
\end{bmatrix}\begin{bmatrix}
L_{11} - I_1 & L_{12}\\
L_{21}  & L_{22} - I_2 
\end{bmatrix}\\
& = \begin{bmatrix}
L_{11} & L_{12}\\
O & I_2
\end{bmatrix},
\end{split}
\end{equation}
whose non-unity eigenvalues are given by the eigenvalues of the submatrix $L_{11}$. Notice the principal submatrix $L_{11}$ is also positive definite, then its eigenvalues have positive real parts \cite{Horn_J2012}. Thus, all the coefficient matrices of Step-(b) in Eq. (\ref{eq3.15}) are nonsingular, which completes the proof.
\end{proof}

\begin{figure}[!htp]
	\begin{center} 
		\includegraphics[width=0.9\textwidth]{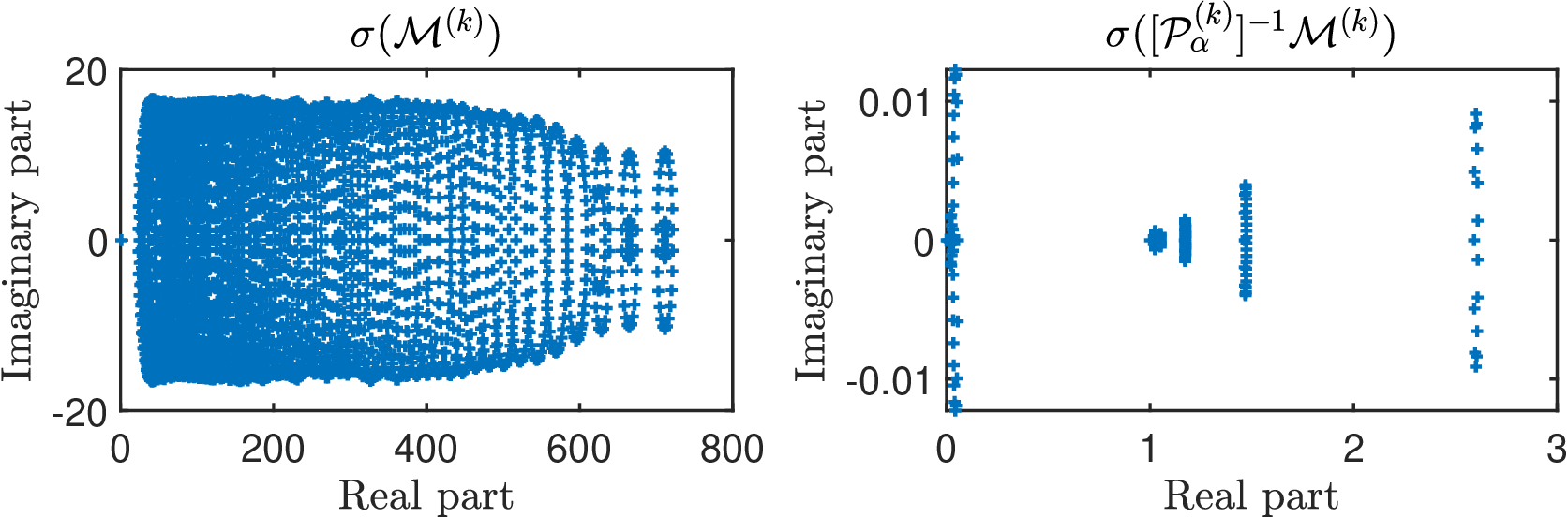}
	\end{center}
	\caption{the spectrum distribution of $\mathcal{M}^{(k)}$ and $[\mathcal{P}_\alpha^{(k)}]^{-1}\mathcal{M}^{(k)}$ in the last ($k=6$) PinT-Policy iteration (Example 1 with $N_s=128$,  $N_t=40$, and $\alpha=10^{-8}$).}	
	\label{FigEx1_nkpa_eigdist}
\end{figure}
Figure \ref{FigEx1_nkpa_eigdist} illustrates the spectrum distribution of $\mathcal{M}^{(k)}$ and the preconditioned matrix $\left[\mathcal{P}_\alpha^{(k)}\right]^{-1}\mathcal{M}^{(k)}$ in the last ($k=6$) PinT-Policy iteration of Example 1. Clearly, the spectrum of preconditioned matrix becomes much more clustered. Nevertheless, there are some eigenvalues near the origin, which indicates less robust convergence rates may be possible, especially in 2D problems.
To provide a glimpse on the approximation property of $\mathcal{P}_\alpha^{(k)}$, we consider the following difference 
\begin{equation} 
\begin{split} 
{\|\mathcal{M}^{(k)} -\mathcal{P}_\alpha^{(k)}\|_F} 
& = {\|\Theta^{(k)}(\mathcal{M} - \mathcal{I}) - (I_t\otimes \Psi^{(k)})(\mathcal{M}_{\alpha} - \mathcal{I})\|_F} \\
& \leq {\|\Theta^{(k)} - (I_t\otimes \Psi^{(k)})\|_F\|\mathcal{M} - \mathcal{I}\|_F + \|I_t\otimes \Psi^{(k)}\|_F\|\mathcal{M} - \mathcal{M}_{\alpha}\|_F},
\end{split}
\end{equation}
which indicates the spectrum analysis of the preconditioned matrices $\left[\mathcal{P}_\alpha^{(k)}\right]^{-1}\mathcal{M}^{(k)}$ is highly nontrivial and further discussion on this issue is beyond the scope of this paper. 

%

\subsubsection{2D model: projection to reduced system.}
According to the above construction of the PinT preconditioner, its key is how to approximate the block diagonal matrix $\Theta^{(k)}$ by a simple Kronecker product. However, the direct application of the above PinT preconditioner seems to work less effectively for solving the 2D models governed by the HJB equation (\ref{eq1.4}). In this subsection, we will propose an alternative approach of constructing PinT preconditioner without introducing any approximation to the 0-1 matrix $\Theta^{(k)}$, which is expected to deliver improved convergence rates.

Given the 0-1 structure of $\Theta^{(k)}$, there is a partitioned permutation matrix $\mathcal{S} = \begin{bmatrix} 
\mathcal{S}_1\\
\mathcal{S}_2
\end{bmatrix}$
such that $$\mathcal{S}\Theta^{(k)}\mathcal{S}^{\top}=\begin{bmatrix}
\mathcal{I}_1\\
&  O 
\end{bmatrix},\qquad \mathcal{S}\mathcal{S}^{\top}=\mathcal{I}=\begin{bmatrix}
\mathcal{I}_1\\
&  \mathcal{I}_2 
\end{bmatrix}$$
where $\mathcal{I}_1$ is an identity matrix collecting all the 1's in $\Theta^{(k)}$.
We first project the coefficient matrices $\mathcal{M}^{(k)}$ with the help of $\mathcal{S}$ as the following (here we omitted superscript $k$ in $\mathcal{S}$ for simplicity)
\begin{equation}
\begin{split}
\mathcal{S}\mathcal{M}^{(k)}\mathcal{S}^{\top}
& = \mathcal{I} + 
\mathcal{S}\Theta^{(k)}\mathcal{S}^{\top}\mathcal{S}(\mathcal{M} - 
\mathcal{I})\mathcal{S}^{\top}\\
& = \begin{bmatrix}
\mathcal{I}_1 \\
& \mathcal{I}_2
\end{bmatrix} + \begin{bmatrix}
\mathcal{I}_1\\
&  O 
\end{bmatrix}\left(\begin{bmatrix}
\mathcal{S}_1\\
\mathcal{S}_2
\end{bmatrix}\mathcal{M}\begin{bmatrix}
\mathcal{S}^{\top}_1 & \mathcal{S}^{\top}_2 
\end{bmatrix} - \begin{bmatrix}
\mathcal{I}_1\\
 & \mathcal{I}_2
\end{bmatrix}\right)\\
& = \begin{bmatrix} 
\mathcal{S}_1\mathcal{M}\mathcal{S}^{\top}_1 & \mathcal{S}_1\mathcal{M}\mathcal{S}^{\top}_2\\
 & \mathcal{I}_2
\end{bmatrix},
\end{split}
\end{equation}
which suggests to solve the following transformed linear system 
\begin{equation}
 \mathcal{S}\mathcal{M}^{(k)}\mathcal{S}^{\top} {\bm y}^{(k)}  = \mathcal{S}{\bm f}^{(k)}, \qquad
{\bm y}^{(k)}  = \mathcal{S}{\bm v}^{(k)},
\end{equation}
i.e., we need to solve 
\begin{equation}
\begin{bmatrix} 
	\mathcal{S}_1\mathcal{M}\mathcal{S}^{\top}_1 & \mathcal{S}_1\mathcal{M}\mathcal{S}^{\top}_2\\
	& \mathcal{I}_2
\end{bmatrix}
\begin{bmatrix}
{\bm y}^{(k)}_1\\
{\bm y}^{(k)}_2
\end{bmatrix} = 
\begin{bmatrix}
{\bm f}^{(k)}_1\\
{\bm f}^{(k)}_2  
\end{bmatrix} ,
\end{equation}
which leads to ${\bm y}^{(k)}_2 = {\bm f}^{(k)}_2$ and the projected linear system of reduced size
\begin{equation} 
	\widetilde{\mathcal{M}} {\bm y}^{(k)}_1:= \left[\mathcal{S}_1\mathcal{M}\mathcal{S}^{\top}_1\right]{\bm y}^{(k)}_1 = 
{\bm f}^{(k)}_1- \left[\mathcal{S}_1\mathcal{M}\mathcal{S}^{\top}_2\right]{\bm y}^{(k)}_2.
\label{eq3.22}
\end{equation}
To better exploit the favorable structure of $\mathcal{M}$ and inspired by the projection idea in \cite{Gould01},  we similarly design the projected PinT preconditioner for the linear system (\ref{eq3.22}):
\begin{equation}
\widetilde{\mathcal{M}}_\alpha^{-1} := \mathcal{S}_1\mathcal{M}^{-1}_{\alpha}\mathcal{S}^{\top}_1.
\end{equation}
When the preconditioned Krylov subspace solver is applied to Eq. (\ref{eq3.22}) with the preconditioner $\widetilde{\mathcal{M}}_\alpha^{-1}$, each iteration of the preconditioned sub-system needs to compute
${\bm z} = \widetilde{\mathcal{M}}_\alpha^{-1}{\bm r}$ via three steps: 
\begin{equation}
\begin{cases}	
	{\bm z}_1 = \left(V^{-1}_{\alpha}\otimes I_x\right)\left(\mathcal{S }^{\top}_1{\bm r}\right) = (\mathbb{F}\otimes I_s)\left[(\Lambda^{-1}_{\alpha}\otimes I_s)\mathcal{S}^{\top}_1{\bm r}\right], &{\rm \text{Step-(a)}},\\
	\left(\lambda_n I_s - L_h \right) {\bm z}_{2,n} = {\bm z}_{1,n},\quad n = 1,2,\cdots,N_t,&{\rm \text{Step-(b)}},\\
	{\bm z} = \mathcal{S}_1(V_{\alpha}\otimes I_s){\bm z}_2
	= \mathcal{S}_1(\Lambda_{\alpha}\otimes I_s)\left[(\mathbb{F}^{*}\otimes
	I_x){\bm z}_2\right],&{\rm \text{Step-(c)}}.
\end{cases}
\label{eq3.24}
\end{equation}
If Propositions \ref{pro3.3}--\ref{pro3.4} hold, it is easy to show that the preconditioner $\widetilde{\mathcal{M}}_\alpha$ is nonsingular (see \cite{lin2021all,gander21paradiag}) and it can be similarly implemented in parallel computations.

The effectiveness of the preconditioner $\widetilde{\mathcal{M}}_\alpha^{-1}$ largely depends on the spectrum distribution of $\widetilde{\mathcal{M}}_\alpha^{-1}\widetilde{\mathcal{M}}$, where a highly clustered spectrum around one usually indicates a fast convergence rate.
We mention that the spectrum of the block $\alpha$-circulant preconditioned matrix $\mathcal{M}_{\alpha}^{-1}\mathcal{M}$ is already known to be uniformly clustered around one  depending only on $\alpha\in (0,1)$.
In Figure \ref{FigEx1_eigdist}, we compared the spectrum distribution of  the preconditioned matrix $\mathcal{M}_{\alpha}^{-1}\mathcal{M}$ and its projected submatrix $\widetilde{\mathcal{M}}_\alpha^{-1}\widetilde{\mathcal{M}}$,
where the spectrum of $\widetilde{\mathcal{M}}_\alpha^{-1}\widetilde{\mathcal{M}}$ seems to be even more clustered  around one. Hence, we anticipate the projected preconditioner $\widetilde{\mathcal{M}}_\alpha^{-1}$ leads to a fast convergence rate as confirmed by the following numerical examples.
\begin{figure}[!htp]
	\begin{center} 
		\includegraphics[width=0.9\textwidth]{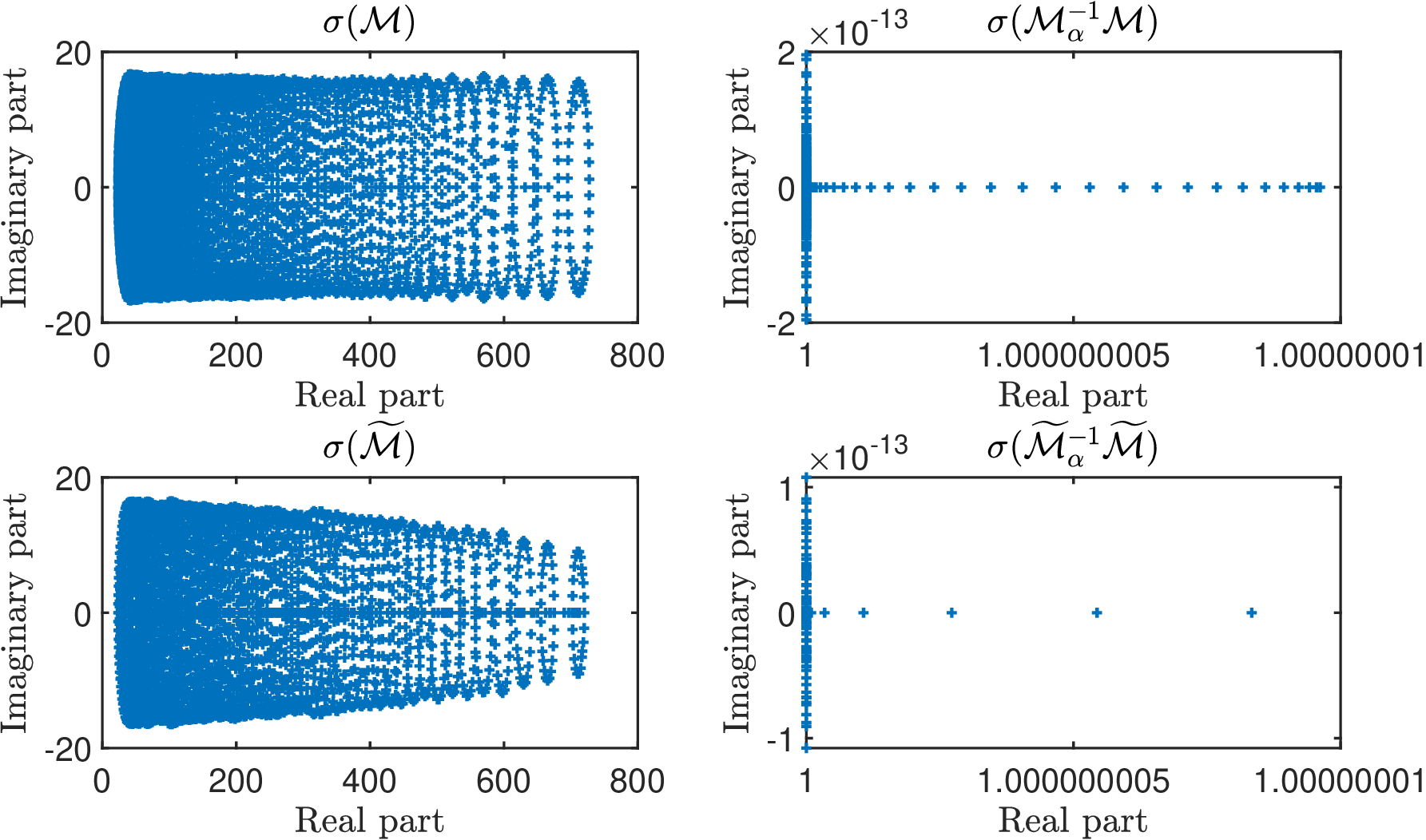}
	\end{center}
	\caption{the spectrum distribution of $\mathcal{M}$ and $\mathcal{M}_{\alpha}^{-1}\mathcal{M}$ and its projected submatrices $\widetilde{\mathcal{M}}$ and $\widetilde{\mathcal{M}}_\alpha^{-1}\widetilde{\mathcal{M}}$ in the last PinT-Policy iteration (Example 1 with $N_s=128$,  $N_t=40$, and $\alpha=10^{-8}$).}	
	\label{FigEx1_eigdist}
\end{figure}
\begin{remark}
Similar to the conjugate technique in \cite{lin2021all}, we find that the complex eigenvalues/eigenvectors of $B^{(\alpha)}$ always appear in conjugate pairs. Thus, we only need to solve for the first $\lceil\frac{N_t+1}{2}\rceil$ systems in Step-(b) of Eq. (\ref{eq3.15}) or Eq. (\ref{eq3.24}). The remaining systems are solved by conjugating the obtained solutions in order to reduce the computational cost by about half.
\end{remark} 
\section{Numerical examples}
\label{sec4}
In this section, we present several numerical examples to illustrate the effectiveness of our proposed PinT-based policy iterations for both 1D and 2D American put option pricing problems.  All simulations are implemented with MATLAB R2024a on a Dell Precision 5820 Workstation with Intel(R) Core(TM) i9-10900X CPU@3.70GHz and 64GB RAM,
where the CPU times (in seconds) are estimated by the timing functions \texttt{tic/toc}.
For the outer all-at-once policy iterations, we choose the initial guess to be the payoff function over all time points, i.e., $\bm v^0=\Phi$, and the stopping criterion based on the absolute residual norm in infinity norm, that is
\[
r_k:=	\left\|\min\left(\mathcal{M} {\bm v^k} - {\bm f}, {\bm v^k} - {\bm \Phi}\right)\right\|_\infty \le tol_1:=10^{-6}.
\]
For the $k$-th inner right preconditioned GMRES iterations, we choose the initial guess to be the approximation $\bm w^0=\bm v^{k-1}$ from the previous policy iteration, and the stopping criterion based on the relative residual norm in Euclidean norm, that is
\[
{\widehat r_l}/{\widehat r_0}\le tol_2:=10^{-10}
\]
with
\[
\widehat r_l:=	\left\|\mathcal{M}^{(k)} \bm w^l- {\bm f}^{(k)}\right\|_2.
\]
Notice that we have selected $tol_2\ll tol_1$ such that the inner linear systems are accurately solved. For the standard policy iterations, the inner linear systems are solved by the highly optimized backslash sparse direct solver within MATLAB, which are computationally expensive due to the all-at-once structure, especially for 2D problems. We will use '--' to indicate out of time limit.

\subsection{Example 1. 1D Black-Scholes Model}
If we consider the B-S model by Black $\&$ Scholes \cite{Black73} and Merton \cite{Merton73}, then it follows that
\begin{equation}
	\begin{split}
		\mathcal{L}V({\bm x},t) &= -\frac{\partial V(s,t)}{\partial t} -
		\frac{\partial^2 V(s,t)}{\partial s^2} - rs\frac{\partial V(s,t)}{\partial s} + rV(s,t)\\
		& := -\frac{\partial V(s,t)}{\partial t} - \mathcal{L}_{BS}V(s,t),\quad (s,t)\in(0,\infty)\times(0,T],
	\end{split}
	\label{eq1.2}
\end{equation}
where $\sigma>0$ is the volatility and $r\ge 0$ is the risk-free interest rate. 
For a non-negative interest rate $r$, the price $u$ of an American put option satisfies the Dirichlet boundary conditions
\begin{equation}
	u(0,t) = K~~{\rm and}~~u(s_{max},t) = 0
\end{equation}
for 1D B-S model. We choose the following model parameters as suggested in \cite{Sydow19}:
\[T=1, K=100, \sigma=0.15, r=0.03, s_{\max}=300.
\]
Figure \ref{FigEx1surface} depicts the typical computed solution surface, which follows the trend of the payoff initial condition.
\begin{figure}[htp!]
	\begin{center}
		\includegraphics[width=1\textwidth]{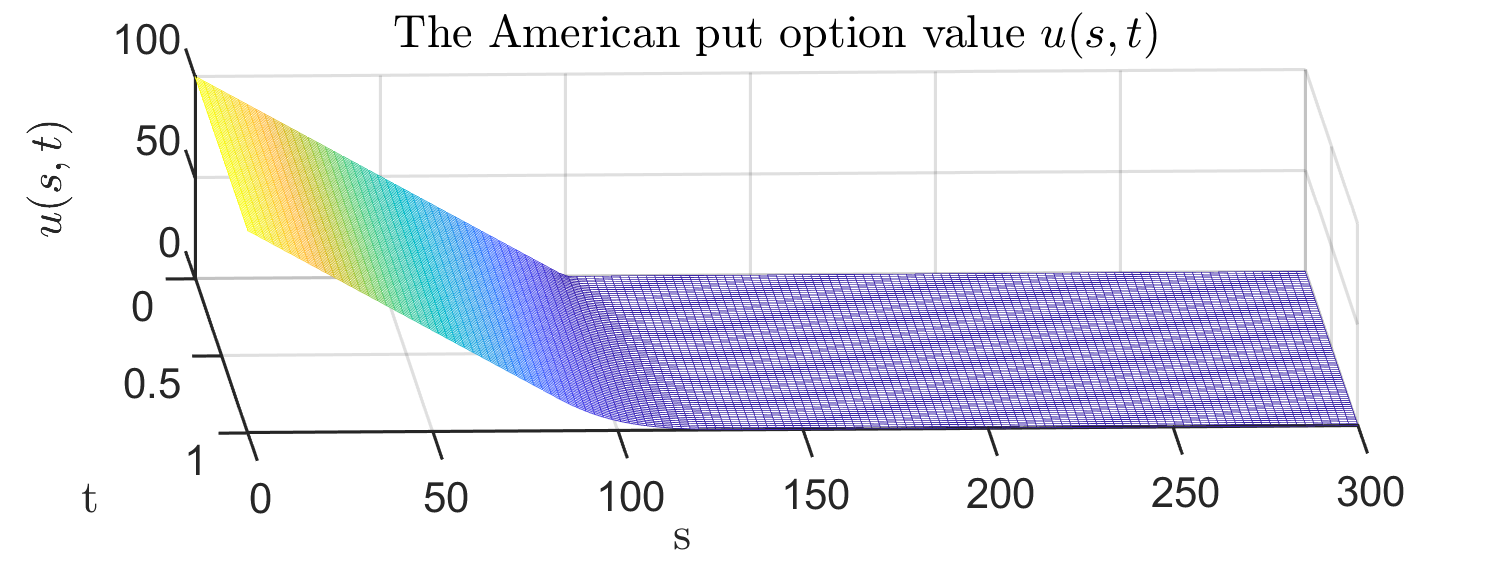} 
	\end{center}
	\caption{1D American put option value surface (with $N_s=128, N_t=80$). }
		\label{FigEx1surface}
\end{figure} 
In Table \ref{Ex1_Table}, we report the errors and iteration numbers of the standard direct Policy iteration and our proposed PinT-Policy iteration with respect to increasing $N_t$ while fixing $N_s$. The `Error' column verifies the expected first-order accuracy due to backward Euler scheme in time. 
We observe the numbers of both Policy iteration are independent of $N_t$, which indicates the strategy of coupling all time steps into an all-at-once large LCP is practical. The inner total preconditioned GMRES iteration number also demonstrates a very robust convergence rate, which illustrates the effectiveness of our proposed PinT preconditioner. The CPU times of PinT-Policy iterations shows a clear linear $\mathcal{O}(N_t)$ complexity for a fixed $N_s$, which is anticipated since the obtained linear systems in Step-(b) can be solved by the fast Thomas tridiagonal solver with a linear complexity.
Based on our MATLAB implementation on a serial PC, our proposed PinT-Policy iteration is about 4-5 times faster than the Policy iteration based on sparse direct solver.
Nevertheless, it is already well-known \cite{Bokanowski09,Reisinger12} that the PinT-Policy iteration numbers grow linearly with respect the spatial mesh size $N_s$, as shown in
Table \ref{Ex1_Table2}, where the average preconditioned GMRES iteration numbers (in parentheses) seem to be mesh-independent, which attributes to our proposed PinT preconditioner.
\begin{table}[!htp]
	\centering
	\caption{Numerical results for Example 1 ($N_s=1280$).
	}
	\begin{tabular}{|c|cc|cc|ccc|}\hline
		&&&\multicolumn{2}{c|}{Policy iteration}  &\multicolumn{3}{c|}{PinT-Policy iteration}
		\\ 	\hline 	
		 $N_t$ & $u(T,K)$ & Error & P-Iter & CPU &  P-Iter & G-Iter & CPU \\ \hline
20& 	 4.771630&	 4.90e-02&	 68&  3.17& 	   68 &191 &   1.16	       \\ \hline
40& 	 4.795074&	 2.55e-02&	 68&  7.83& 	   68 &224 &   2.34  	       \\ \hline
80& 	 4.807377&	 1.32e-02&	 68&  19.60&    68 &230 &   5.14  	       \\ \hline	
160& 	 4.813769&	 6.84e-03&	 68&  49.83&    68 &236 &   10.92 	       \\ \hline	
320& 	 4.817070&	 3.54e-03&	 68&  123.23&   68 &237&   21.26 	       \\ \hline	
640& 	 4.818767&	 1.84e-03&	  -- &  --   &      68 &239&   42.39 	       \\ \hline	
1280& 	 4.819637&	 9.71e-04&	 --  &  --   &     68 &239 &   85.54 	       \\ \hline	
		\hline
		Ref. \cite{Sydow19} & 4.820608 &&&&&&\\
		\hline
	\end{tabular}
	\label{Ex1_Table}
\end{table}
\begin{table}[!htp]
	\centering
	\caption{PinT-Policy and average GMRES iteration numbers for Example 1.
	}
	\begin{tabular}{|c|c|c|c|c|c|c|c|}\hline
		$N_t\backslash N_s$& 20 & 40 & 80 &160 &320 &640 &1280
		\\ 	\hline 	
  20&  	 2 (3.50)     & 	 2 (4.00)     & 	 4 (4.00)      & 	 8 (3.25)     & 	 17 (3.41)    & 	 34 (3.21)  &    68 (3.19)  \\ \hline
  40&  	 2 (4.00)     & 	 2 (4.00)     & 	 4 (4.25)      & 	 8 (3.50)     & 	 17 (3.65)    & 	 34 (3.44)  &   68 (3.31)	\\ \hline
  80&  	 2 (4.00)     & 	 2 (4.00)     & 	 4 (4.25)      & 	 8 (3.88)     & 	 17 (3.94)    & 	 34 (3.62)  &   68 (3.41)	\\ \hline
  160& 2 (4.00)         & 	 2 (4.00)     & 	 4 (4.25)      & 	 8 (4.00)     & 	 17 (4.06)    & 	 34 (3.68)  &    68 (3.51)  \\ \hline
  320& 2 (4.00)         & 	 2 (4.00)     & 	 4 (4.25)      & 	 8 (4.00)     & 	 17 (4.12)    & 	 34 (3.76)  &    68 (3.53)  \\ \hline
  640& 2 (4.00)         & 	 2 (4.00)     & 	 4 (4.50)      & 	 8 (4.12)     & 	 17 (4.18)    &  34 (3.82)      &    68 (3.57)  \\ \hline
  1280&  2 (4.00)       & 	 2 (4.00)     & 	 4 (4.50)      & 	 8 (4.25)     & 	 17 (4.24)    &  34 (3.82)      &    68 (3.59)  \\ \hline
	\end{tabular}
	\label{Ex1_Table2}
\end{table}
 \begin{figure}[htp!]
	\begin{center}
		\includegraphics[width=0.9\textwidth]{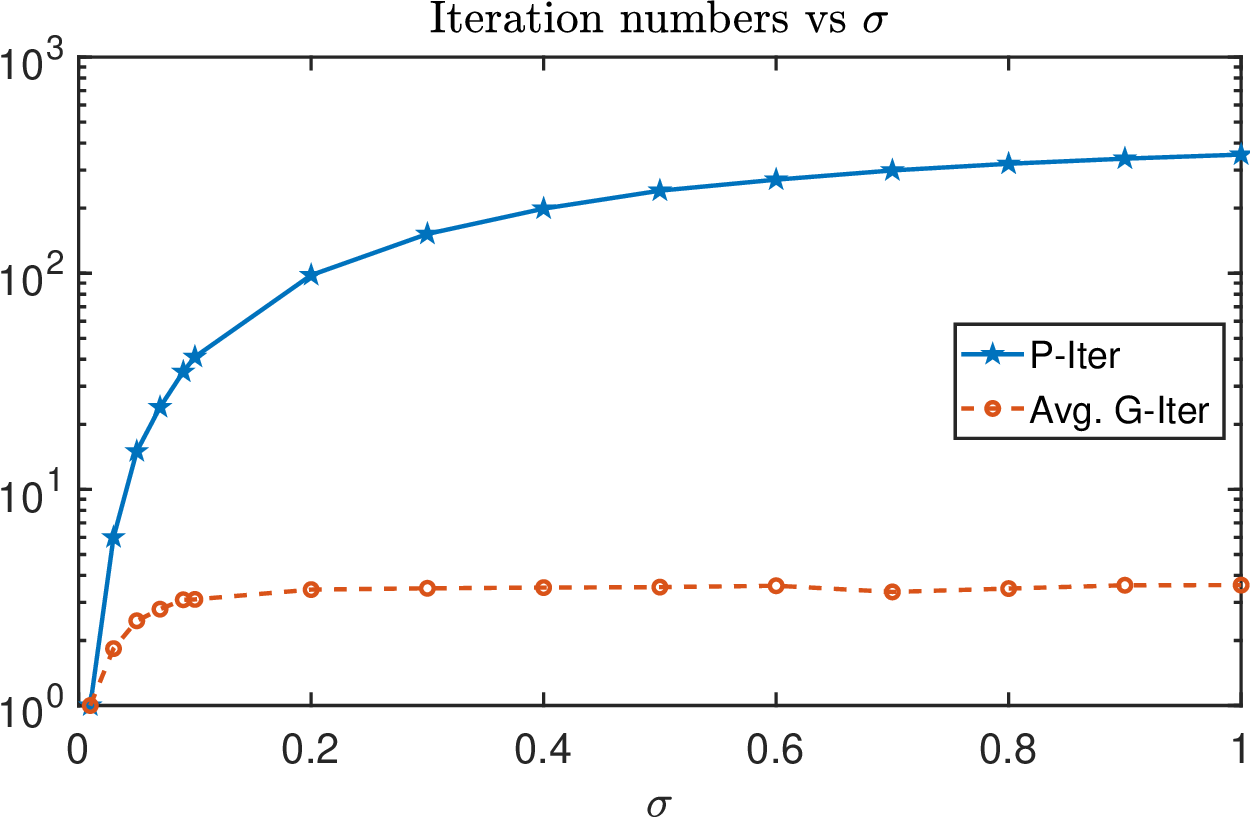}
	\end{center}
	\caption{Example 1: PinT-Policy iteration numbers verse $\sigma$ (with $N_s=1280, N_t=80$). }
	
	\label{FigEx1}
\end{figure}

In Figure \ref{FigEx1}, we plot the PinT-Policy iteration numbers verse different values of $\sigma\in [0.01,1]$ while fixing the other parameters. Clearly, the outer policy iteration number grows mildly, while the inner average preconditioned GMRES iteration number (about 4) seems to very robust whenever $\sigma\ge 0.1$.

\subsection{Example 2: 2D American put spread options}
In this example, we consider the two-asset B-S American spread option model given by
\begin{equation}
	\begin{split}
		\mathcal{L}V({\bm x},t) &= -\frac{\partial V(s,v,t)}{\partial t} - \frac{1}{2}\sigma^{2}_1s^2\frac{\partial^2 V(s,v,t)}{\partial s^2} - \frac{1}{2}\sigma^{2}_2v^2\frac{\partial^2 V(s,v,t)}{\partial v^2}~- \\
		&\quad\ \rho\sigma_1 \sigma_2sv\frac{\partial^2 V(s,v,t)}{\partial s\partial v} 
		- rs\frac{\partial V(s,v,t)}{\partial s} - rv\frac{\partial V(s,v,t)}{\partial v} + rV(s,v,t)\\
		&:= -\frac{\partial V(s,v,t)}{\partial t} - \mathcal{L}_{S}V(s,v,t)
	\end{split}
\end{equation}
for $s,v > 0, t\in[0,T)$. Where $\sigma_i$ ($i = 1,2)$ are the deterministic local volatility of each asset, $r$ is the risk-free interest rate, 
$\rho\in [-1, 0)\cup(0, 1]$ is the correlation of two underlying
assets. Under the two-asset B-S model, boundary conditions for American spread put options
\begin{eqnarray*}
	u(0,v,t) = \max\{Ke^{-rt}+v,0\} = 0,& u(s,0,t) = \max\{Ke^{-rt}-s,0\},\\
	u(s_{max},v,t) = \max\{Ke^{-rt}-(s_{max}-v),0\},&
	u(s,v_{max},t) = \max\{Ke^{-rt} - (s - v_{max}),0\}
\end{eqnarray*}
are posed according to Ref. \cite{heidarpour2018}. Moreover, there are some other kinds of options and BCs introduced in \cite{Chen2017}, but our proposed method can similarly deal with them and thus we shall not pursue that here. We choose the model parameters as suggested in \cite{heidarpour2018}:
$$T = \frac{122}{365},~K=25,~\sigma_1=0.35,~\sigma_2=0.38,~r=0.035,~ \rho=0.6,s_{\max}=v_{\max}=300.$$

Figure \ref{FigEx1surface} plots the typical computed solution surface at the final time $T$, where the spot point $(s^*,v^*)=(127.68,99.43)$ is marked out.
In Table \ref{Ex2_Table}, we report the errors and iteration numbers of the original direct Policy iteration and our proposed PinT-Policy iteration with respect to increasing $N_t$ while fixing $N_s$. We observe the numbers of direct Policy iteration are independent of $N_t$, which indicates the strategy of coupling all time steps into an all-at-once large LCP is practical. The inner total preconditioned GMRES iteration number also demonstrates a very robust convergence rate, which illustrates the effectiveness of our proposed PinT preconditioner.
To reduce the computational cost of Step-(b), we only approximately solve the linear systems in Step-(b) by one geometric multgrid V-cycle with the ILU smoother, which can greatly speed up the CPU times as reported in Table \ref{Ex2_Table}. For example, with $N_t=160$, the CPU time reduces from 564 seconds to 10 seconds. Without fast solvers for Step-(b), the speedup will be less dramatic in serial implementation, as shown in the next example.
 \begin{figure}[htp!]
	\begin{center}
		\includegraphics[width=0.9\textwidth]{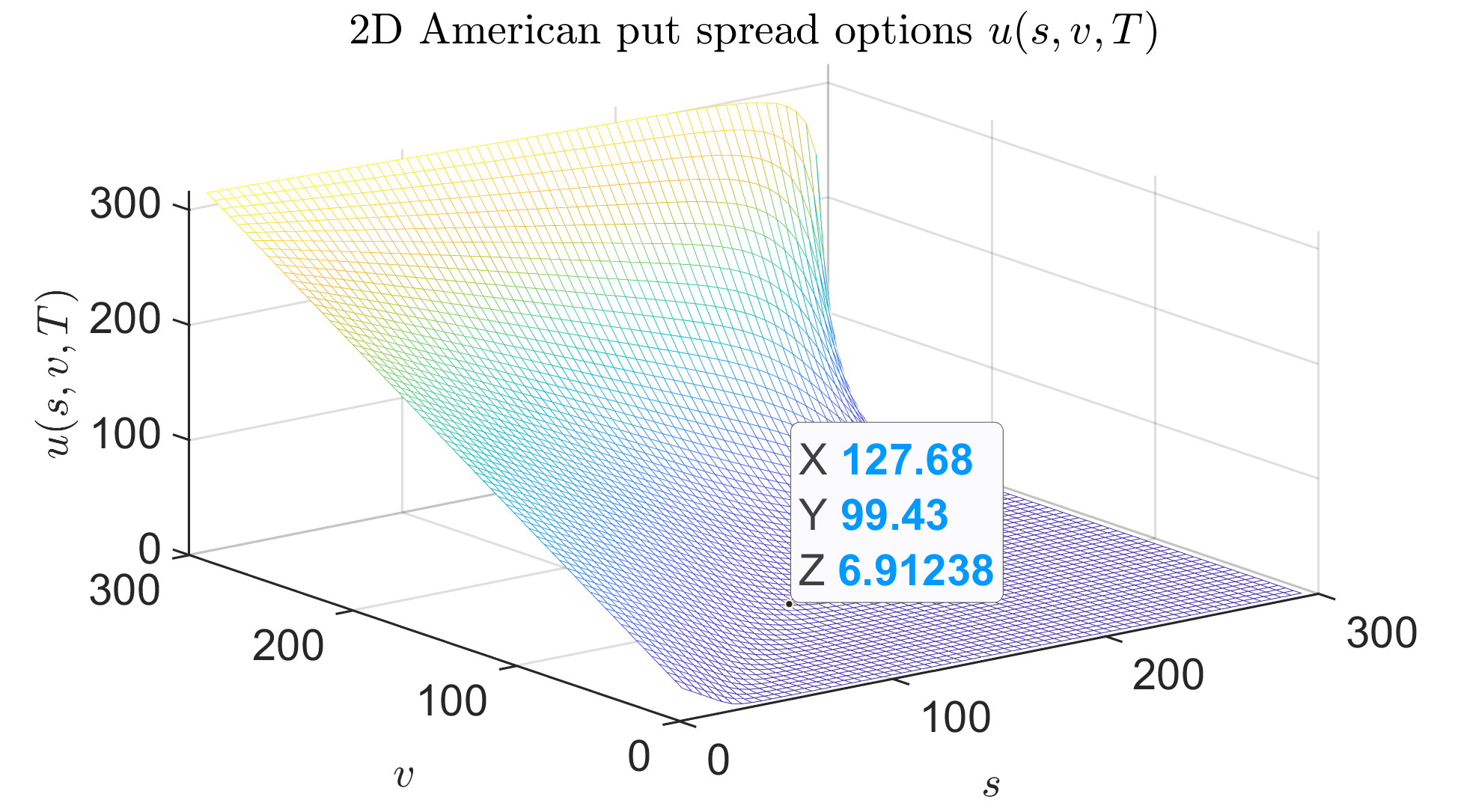} 
	\end{center}
	\caption{2D American put spread option value at $T$ (with $N_s=64\times 64, N_t=80$). }
	\label{FigEx2surface}
\end{figure}

 \begin{table}[htp!]
 	\centering
 	\caption{Numerical results for Example 2 ($N_s=N_v=64$).}
 	\begin{tabular}{|c|cc|cc|ccc|}\hline
 		&&&\multicolumn{2}{c|}{Policy iteration}  &\multicolumn{3}{c|}{PinT-Policy iteration}
 		\\ 	\hline 	
 		$N_t$ & $u(T,s^*,v^*)$ & Error & P-Iter & CPU &  P-Iter & G-Iter & CPU \\ \hline
 	 10& 	 6.820604&	 1.12e-01&	      5 &     2.77&	            5&  51&         0.72 	\\ \hline
 	 20& 	 6.872221&	 6.07e-02&	      5 &     10.26&	        6&  52&         1.27 	\\ \hline
 	 40& 	 6.898814&	 3.41e-02&	      6 &     45.61&	        6&  57&         2.61 	\\ \hline
 	 80& 	 6.912376&	 2.05e-02&	      6 &     163.87&           6&  59&         5.28 	\\ \hline
 	 160& 	 6.919245&	 1.36e-02&	      8  &    564.10&           6&  60&       10.21      \\ \hline
 	 320& 	 6.922706&	 1.02e-02&	      --  &     --    &           6&  61&       20.61      \\ \hline
 	 640& 	 6.924444&	 8.43e-03&	      -- &      --   &           6&  61&      42.57 	    \\ \hline
 		\hline
 		Ref.\cite{heidarpour2018} & 6.932875 &&&&&&\\
 		\hline
 	\end{tabular}
 	\label{Ex2_Table}
 \end{table}

\subsection{Example 3. 2D Heston Model.}
In Heston model, the stock price process $S_t$ and the variance process $V_t$ are defined by the specific stochastic differential equations. Let $V(s,v,t)$ be the value of an American option under the Heston model with striking price $K$, the corresponding HJB equation (\ref{eq1.1}) enjoys the following component
\begin{equation}
\begin{split}
\mathcal{L}V({\bm x},t) & = -\frac{\partial V(s,v,t)}{\partial t} - \frac{1}{2}vs^2\frac{\partial^2 V(s,v,t)}{\partial s^2} - \frac{1}{2}\sigma^{2}v\frac{\partial^2 V(s,v,t)}{\partial v^2} - \rho\sigma sv\frac{\partial^2 V(s,v,t)}{\partial s\partial v} \\
&\quad\ - rs\frac{\partial V(s,v,t)}{\partial s} - \kappa(\eta - v)\frac{\partial V(s,v,t)}{\partial v} + rV(s,v,t)\\
& := -\frac{\partial V(s,v,t)}{\partial t} - \mathcal{L}_{H}V(s,v,t).
\label{eq1.5x}
\end{split}
\end{equation}
We deal with boundary conditions of Dirichlet and Neumann type, determined by the specific option under consideration, or no condition, in the case of a degenerate boundary. For a vanilla American put option, the following boundary conditions are common in the literature.
\begin{equation}
\begin{cases}
u(0,v,t) = K, \\
u_s(s_{max},v,t) = 0,\\
u_v(s,v_{max},t) = 0. \\ 
\end{cases} 
\end{equation}
Note the degeneracy feature of Heston operator (\ref{eq1.5x}) in the $v$-direction since all second-order derivatives vanish and the operator becomes convection-dominated for $v\downarrow 0$. Hence, at $v=0$, it is assumed the Heston LCP (\ref{eq1.1}) is fulfilled \cite{Haentjens15}. We select the model parameters as suggested in \cite{Haentjens15}
$$T=0.25,~K=10,~\sigma=0.9,~r=0.1,~\kappa=5,~\eta=0.16,~\rho=0.1, s_{\max}=20, v_{\max}=1.$$
Figure \ref{FigEx3surface} plots the typical computed solution surface at the final time $T$, where the spot point $(K,v_0)=(10,0.25)$ is marked out.
In Table \ref{Ex3_Table}, we report the errors and iteration numbers of the original direct Policy iteration and our proposed PinT-Policy iteration with respect to increasing $N_t$ while fixing $N_s$. Again, we observe the numbers of direct Policy iteration are independent of $N_t$, and the inner total preconditioned GMRES iteration number indeed demonstrates a very mild increase, which verifies the effectiveness of our proposed PinT preconditioner.
We point out the spatial differential operator in Heston model leads to indefinite linear systems in Step-(b) and hence the geometric multigrid solver can not be directly applied. Hence, in this case we observe only marginal speed up in serial implementation. Nevertheless, we believe the speed up will become more plausible in parallel computation since Step-(b) contributes about 90\% of the total CPU times.
\begin{figure}[htp!]
	\begin{center}
		\includegraphics[width=0.9\textwidth]{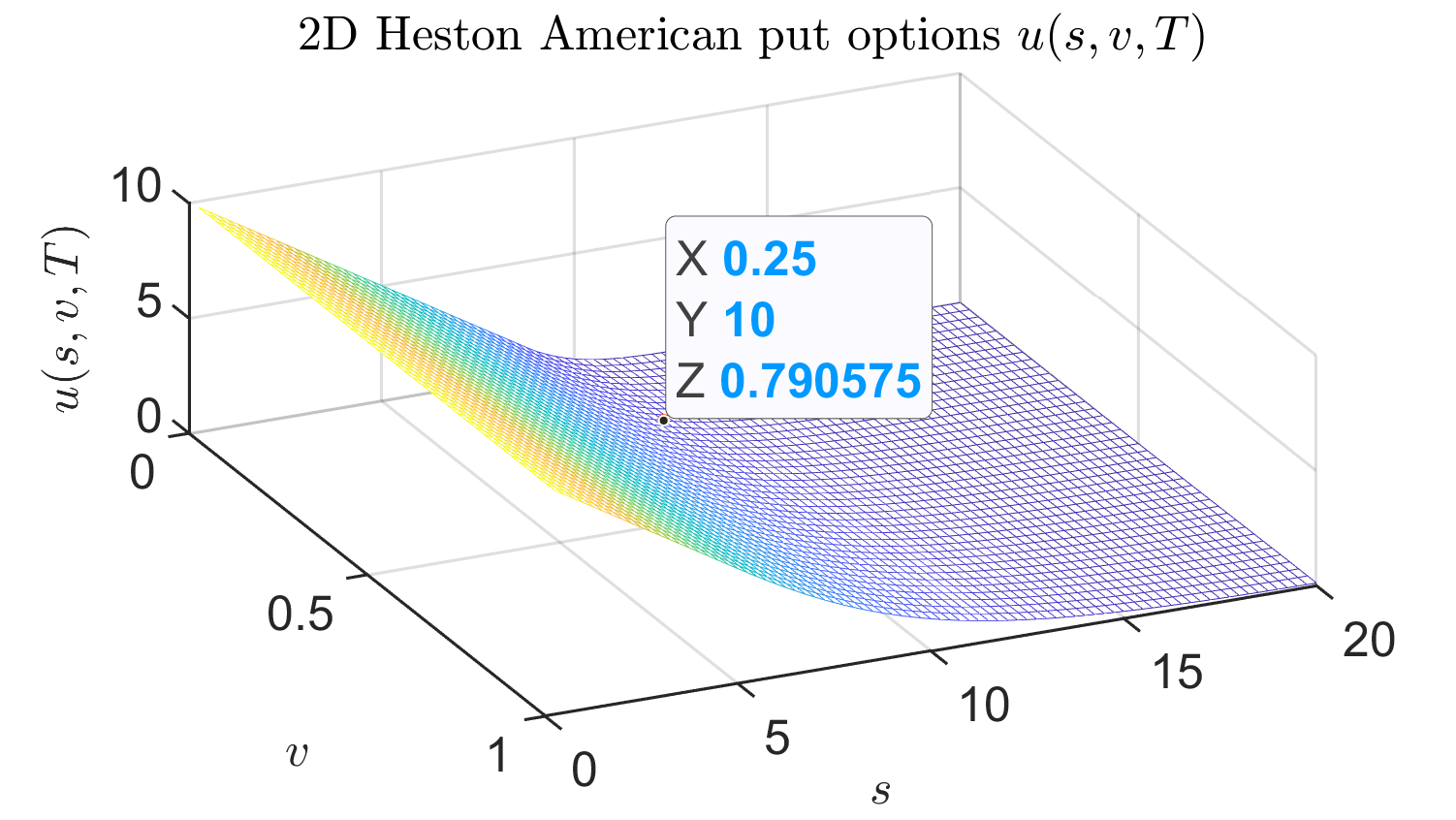} 
	\end{center}
	\caption{2D Heston American put option value at $T$ (with $N_s=80, N_v=40, N_t=40$). }
	\label{FigEx3surface}
\end{figure}

\begin{table}[!htp]
	\centering
	\caption{Numerical results for Example 3 ($N_s=80, N_v=40$, $v_0=0.25$). 
	}
	\begin{tabular}{|c|cc|cc|ccc|}\hline
		&&&\multicolumn{2}{c|}{Policy iteration}  &\multicolumn{3}{c|}{PinT-Policy iteration}
		\\ 	\hline 	
		$N_t$ & $u(K,v_0,T)$ & Error & P-Iter & CPU &  P-Iter & G-Iter & CPU \\ \hline
		10& 	0.780152&	1.58e-02&    8  &	 4.32&	      8 &89 &	     4.47 	  \\\hline
		20& 	0.786990&	8.98e-03&    8  &	 17.12&	      8 &89 &	     8.32	  \\\hline
		40& 	0.790575&	5.39e-03&    8  &	 57.81&	      8 &93 &	      23.66	  \\\hline
		80& 	0.792432&	 3.54e-03&    8  &	 184.76&	   8 &96 &     85.36   \\\hline
		160& 	0.793386&	 2.58e-03&    8  &	 566.13&	   8 &96 &     311.16   \\\hline
		\hline
		Ref.\cite{Haentjens15} & 0.795968 &&&&&&\\
		\hline
	\end{tabular}
	\label{Ex3_Table}
\end{table}


\section{Conclusion}
\label{sec5}
For pricing various American options based on different PDE
models, the conventional numerical approaches need to solve a sequence of linear complementarity problems (LCPs) at
each time step sequentially. One well-known way of solving such LCPs is to to reformulate them as HJB equations, which can be
then solved by the policy iteration with fast convergence rate. 
To utilizing the modern parallel computing power, we propose to solve an “all-at-once” form of
HJB equations simultaneously by the policy iteration, which can be accelerated by our carefully
designed parallel-in-time (PinT) preconditioners. Numerical examples are presented
to confirm the effectiveness of our proposed algorithms.
Although the numerical performance is very convincing, solid theoretical analysis for supporting our numerical observations is still widely open to the community. In particular, it is crucial to prove the policy iteration numbers are independent of the increasing $N_t$.
\subsection*{Acknowledgements}
The first author would like to thank Dr. Yong-Liang Zhao (Sichuan Normal University) for helping the numerical discretizations used in Example 3 of Section \ref{sec4}.

\subsection*{Funding} Xian-Ming Gu's research is supported by the Sichuan Science and Technology Program (2022ZYD0006) and the Guanghua Talent
Project of Southwestern University of Finance and Economics. 
\subsection*{Data Availibility} The developed codes are freely available upon reasonable request.
\subsection*{Declarations}
The authors have no relevant financial or non-financial interests to disclose.
\bibliographystyle{siamplain}
\bibliography{Refxx}

\begin{thebibliography}{10}

\bibitem{arregui2017pde}
{\sc I.~Arregui, B.~Salvador, and C.~V{\'a}zquez}, {\em {PDE} models and
  numerical methods for total value adjustment in {European and American}
  options with counterparty risk}, Appl. Math. Comput., 308 (2017), pp.~31--53.

\bibitem{Berman94}
{\sc A.~Berman and R.~J. Plemmons}, {\em Nonnegative Matrices in the
  Mathematical Sciences}, SIAM, Philadelphia, PA, 1994,
  \url{https://doi.org/10.1137/1.9781611971262}.

\bibitem{bini2005}
{\sc D.~Bini, G.~Latouche, and B.~Meini}, {\em Numerical Methods for Structured
  Markov Chains}, Numerical Mathematics and Scientific Computation, Oxford
  University Press, Oxford, UK, 2005.

\bibitem{Black73}
{\sc F.~Black and M.~Scholes}, {\em The pricing of options and corporate
  liabilities}, J. Political Economy, 81 (1973), pp.~637--654,
  \url{https://doi.org/10.1086/260062}.

\bibitem{Bokanowski09}
{\sc O.~Bokanowski, S.~Maroso, and H.~Zidani}, {\em Some convergence results
  for {Howard's} algorithm}, SIAM J. Numer. Anal., 47 (2009), pp.~3001--3026,
  \url{https://doi.org/10.1137/08073041X}.

\bibitem{borovkova2012american}
{\sc S.~A. Borovkova, F.~J. Permana, and J.~A.~M. Van Der~Weide}, {\em American
  basket and spread option pricing by a simple binomial tree}, J. Deriv., 19
  (2012), pp.~29--38.

\bibitem{broadie1996american}
{\sc M.~Broadie and J.~Detemple}, {\em American option valuation: new bounds,
  approximations, and a comparison of existing methods}, Rev. Financ. Stud., 9
  (1996), pp.~1211--1250.

\bibitem{Cen2010}
{\sc Z.~Cen and A.~Le}, {\em A robust finite difference scheme for pricing
  american put options with singularity-separating method}, Numer. Algorithms,
  53 (2010), \url{https://doi.org/10.1007/s11075-009-9316-x}.

\bibitem{Chen2017}
{\sc Y.~Chen}, {\em {Numerical Methods for Pricing Multi-Asset Options}},
  {Master} thesis, Department of Computer Science, University of Toronto,
  Toronto, Canada, 2017.

\bibitem{Clevenhaus22}
{\sc A.~Clevenhaus, M.~Ehrhardt, and M.~G{\"u}nther}, {\em The parareal
  algorithm and the sparse grid combination technique in the application of the
  heston model}, in Progress in Industrial Mathematics at ECMI 2021,
  M.~Ehrhardt and M.~G{\"u}nther, eds., Springer International Publishing,
  Cham, 2022, pp.~477--483.

\bibitem{fang2009pricing}
{\sc F.~Fang and C.~W. Oosterlee}, {\em Pricing early-exercise and discrete
  barrier options by {Fourier-cosine} series expansions}, Numer. Math., 114
  (2009), pp.~27--62.

\bibitem{forsyth2007numerical}
{\sc P.~A. Forsyth and G.~Labahn}, {\em Numerical methods for controlled
  {Hamilton-Jacobi-Bellman PDEs} in finance}, J. Comput. Finance, 11 (2007),
  p.~1.

\bibitem{forsyth2002quadratic}
{\sc P.~A. Forsyth and K.~R. Vetzal}, {\em Quadratic convergence for valuing
  {American} options using a penalty method}, SIAM J. Sci. Comput., 23 (2002),
  pp.~2095--2122.

\bibitem{gander21paradiag}
{\sc M.~J. Gander, J.~Liu, S.-L. Wu, X.~Yue, and T.~Zhou}, {\em {ParaDiag}:
  parallel-in-time algorithms based on the diagonalization technique}, 2021,
  \url{https://arxiv.org/abs/2005.09158v4}.
\newblock arXiv preprint.

\bibitem{golub2013matrix}
{\sc G.~H. Golub and C.~F. Van~Loan}, {\em Matrix Computations}, Johns Hopkins
  University Press, Baltimore, MD, 4~ed., 2013.

\bibitem{Gould01}
{\sc N.~I.~M. Gould, M.~E. Hribar, and J.~Nocedal}, {\em On the solution of
  equality constrained quadratic programming problems arising in optimization},
  SIAM J. Sci. Comput., 23 (2001), pp.~1376--1395,
  \url{https://doi.org/10.1137/S1064827598345667}.

\bibitem{Hackbusch17}
{\sc W.~Hackbusch}, {\em Elliptic Differential Equations: Theory and Numerical
  Treatment}, vol.~18 of Springer Series in Computational Mathematics, Springer
  Berlin, Heidelberg, 2nd~ed., 2017.

\bibitem{Haentjens15}
{\sc T.~Haentjens and K.~J. in~’t Hout}, {\em {ADI} schemes for pricing
  {American} options under the {Heston} model}, Appl. Math. Finan., 22 (2015),
  pp.~207--237, \url{https://doi.org/10.1080/1350486X.2015.1009129}.

\bibitem{heidarpour2018}
{\sc V.~Heidarpour-Dehkordi and C.~Christara}, {\em Spread option pricing using
  {ADI} methods}, Int. J. Numer. Anal. Mod., 15 (2018), pp.~353--369.

\bibitem{Horn_J2012}
{\sc R.~A. Horn and C.~R. Johnson}, {\em Matrix Analysis}, Cambridge University
  Press, Cambridge, UK, 2~ed., 2012.

\bibitem{hout2016application}
{\sc K.~i. Hout and J.~Toivanen}, {\em Application of operator splitting
  methods in finance}, in Splitting Methods in Communication, Imaging, Science,
  and Engineering, R.~Glowinski, S.~J. Osher, and W.~Yin, eds., Springer
  International Publishing, Cham, 2016, pp.~541--575,
  \url{https://doi.org/10.1007/978-3-319-41589-5_16}.

\bibitem{ikonen2008efficient}
{\sc S.~Ikonen and J.~Toivanen}, {\em Efficient numerical methods for pricing
  {American} options under stochastic volatility}, Numer. Meth. Part Differ.
  Equ., 24 (2008), pp.~104--126.

\bibitem{itkin2017}
{\sc A.~Itkin}, {\em Pricing Derivatives Under L{\'e}vy Models: Modern
  Finite-Difference and Pseudo-Differential Operators Approach},
  Pseudo-Differential Operators, Birkh\"{a}user, New York, NY, 2017.

\bibitem{jaillet1990variational}
{\sc P.~Jaillet, D.~Lamberton, and B.~Lapeyre}, {\em Variational inequalities
  and the pricing of {American} options}, Acta Appl. Math., 21 (1990),
  pp.~263--289.

\bibitem{Le2012}
{\sc A.~Le, Z.~Cen, and A.~Xu}, {\em A robust upwind difference scheme for
  pricing perpetual {American} put options under stochastic volatility}, Int.
  J. Comput. Math., 89 (2012), pp.~1135--1144,
  \url{https://doi.org/10.1080/00207160.2012.658379}.

\bibitem{lin2021all}
{\sc X.-L. Lin and M.~Ng}, {\em An all-at-once preconditioner for evolutionary
  partial differential equations}, SIAM J. Sci. Comput., 43 (2021),
  pp.~A2766--A2784, \url{https://doi.org/10.1137/20M1316354}.

\bibitem{Liu2022}
{\sc J.~Liu and S.-L. Wu}, {\em Parallel-in-time preconditioner for the
  {Sinc-Nyström} systems}, SIAM J. Sci. Comput., 44 (2022), pp.~A2386--A2411,
  \url{https://doi.org/10.1137/21M1462696}.

\bibitem{longstaff2001valuing}
{\sc F.~A. Longstaff and E.~S. Schwartz}, {\em Valuing {American} options by
  simulation: a simple least-squares approach}, Rev. Financ. Stud., 14 (2001),
  pp.~113--147.

\bibitem{Maday2008}
{\sc Y.~Maday and E.~M. Rønquist}, {\em Parallelization in time through
  tensor-product space–time solvers}, C. R. Acad. Sci. Paris, Ser. I, 346
  (2008), pp.~113--118, \url{https://doi.org/10.1016/j.crma.2007.09.012}.

\bibitem{maree2017pricing}
{\sc S.~C. Maree, L.~Ortiz-Gracia, and C.~W. Oosterlee}, {\em Pricing
  early-exercise and discrete barrier options by {Shannon} wavelet expansions},
  Numer. Math., 136 (2017), pp.~1035--1070.

\bibitem{Merton73}
{\sc R.~C. Merton}, {\em Theory of rational option pricing}, Bell J. Econom.
  Management Sci., 4 (1973), pp.~141--183,
  \url{https://doi.org/10.2307/3003143}.

\bibitem{Mrkaic02}
{\sc M.~Mrkaic}, {\em Policy iteration accelerated with {Krylov} methods}, J.
  Econ. Dyn. Control, 26 (2002), pp.~517--545,
  \url{https://doi.org/10.1016/S0165-1889(00)00073-7}.

\bibitem{oosterlee2003multigrid}
{\sc C.~W. Oosterlee}, {\em On multigrid for linear complementarity problems
  with application to {American-style} options}, Electron. Trans. Numer. Anal.,
  15 (2003), pp.~165--185.

\bibitem{Pham1997}
{\sc H.~Pham}, {\em Optimal stopping, free boundary, and {American} option in a
  jump-diffusion model}, Appl. Math. Optim., 35 (1997), pp.~145--164,
  \url{https://doi.org/10.1007/BF02683325}.

\bibitem{Reisinger12}
{\sc C.~Reisinger and J.~H. Witte}, {\em On the use of policy iteration as an
  easy way of pricing {American} options}, SIAM J. Finan. Math., 3 (2012),
  pp.~459--478, \url{https://doi.org/10.1137/110823328}.

\bibitem{saad2003iterative}
{\sc Y.~Saad}, {\em Iterative Methods for Sparse Linear Systems}, SIAM,
  Philadelphia, PA, 2nd~ed., 2003.

\bibitem{Shen2023}
{\sc J.~Shen, W.~Huang, and J.~Ma}, {\em An efficient and provable sequential
  quadratic programming method for {American} and swing option pricing}, Eur.
  J. Oper. Res., 316 (2024), pp.~19--35,
  \url{https://doi.org/10.1016/j.ejor.2023.11.012}.

\bibitem{VanLoan1993}
{\sc C.~F. Van~Loan and N.~Pitsianis}, {\em Approximation with {Kronecker}
  products}, in Linear Algebra for Large Scale and Real-Time Applications,
  M.~S. Moonen, G.~H. Golub, and B.~L.~R. De~Moor, eds., Springer Netherlands,
  Dordrecht, 1993, pp.~293--314,
  \url{https://doi.org/10.1007/978-94-015-8196-7_17}.

\bibitem{vellekoop2009tree}
{\sc M.~Vellekoop and H.~Nieuwenhuis}, {\em A tree-based method to price
  {American} options in the heston model}, J. Comput. Finance, 13 (2009),
  pp.~1--21.

\bibitem{Sydow19}
{\sc L.~von Sydow, S.~Milovanovi\'{c}, E.~Larsson, K.~In't~Hout, M.~Wiktorsson,
  C.~W. Oosterlee, V.~Shcherbakov, M.~Wyns, A.~Leitao, S.~Jain, H.~T., and
  J.~Wald\'{e}n}, {\em {BENCHOP -- SLV}: the {BENCHmarking} project in option
  pricing -- stochastic and local volatility problems}, Int. J. Comput. Math.,
  96 (2019), pp.~1910--1923,
  \url{https://doi.org/10.1080/00207160.2018.1544368}.

\bibitem{Wilmott1993}
{\sc P.~Wilmott, J.~N. Dewynne, and S.~Howison}, {\em Option Pricing:
  Mathematical Models and Computation}, Oxford Financial Press, Oxford, UK,
  1993.

\bibitem{zvan1998penalty}
{\sc R.~Zvan, P.~A. Forsyth, and K.~R. Vetzal}, {\em Penalty methods for
  {American} options with stochastic volatility}, J. Comput. Appl. Math., 91
  (1998), pp.~199--218.

\end{thebibliography}
\end{document}